\documentclass[11pt,thmsa]{article}%
\usepackage{amsmath}
\usepackage{amsfonts}
\usepackage{amssymb}
\usepackage{graphicx}%
\newtheorem{theorem}{Theorem}[section]

\newtheorem{corollary}[theorem]{Corollary}

\newtheorem{definition}[theorem]{Definition}

\newtheorem{lemma}[theorem]{Lemma}

\newenvironment{proof}[1][Proof]{\noindent\textbf{#1.} }{\ \rule{0.5em}{0.5em}}

\begin{document}
\title{Finsler spheres with constant flag curvature and finite
orbits of prime closed geodesics\thanks{Supported by NSFC (no. 11771331) and NSF of Beijing (no. 1182006)}}
\author{Ming Xu\\
\\
School of Mathematical Sciences\\
Capital Normal University\\
Beijing 100048, P. R. China\\
Email: mgmgmgxu@163.com
\\
}
\date{}
\maketitle

\begin{abstract}
In this paper, we consider a Finsler sphere $(M,F)=(S^n,F)$ with the dimension $n>1$ and the flag curvature $K\equiv 1$.
The action of the connected isometry group $G=I_o(M,F)$ on $M$, together with the action of $T=S^1$ shifting the parameter $t\in \mathbb{R}/\mathbb{Z}$
of the closed curve $c(t)$, define an action of $\hat{G}=G\times T$ on the free loop space $\lambda M$ of $M$. In particular, for each closed geodesic, we have a $\hat{G}$-orbit of closed geodesics. We assume the Finsler sphere $(M,F)$ described above has only finite orbits of
prime closed geodesics. Our main theorem claims, if the
subgroup $H$ of all isometries preserving each close geodesic
has a dimension $m$, then there exists $m$ geometrically distinct orbits $\mathcal{B}_i$ of prime closed geodesics, such that for each $i$,
the union ${B}_i$ of geodesics in $\mathcal{B}_i$ is a totally
geodesic sub-manifold in $(M,F)$ with a non-trivial $H_o$-action.
This theorem generalizes and slightly refines the one in a previous
work, which only discussed the case of finite prime closed geodesics.
At the end, we show that, assuming certain generic conditions, the Katok metrics, i.e. the Randers metrics on spheres with $K\equiv 1$, provide examples with the sharp estimate for
our main theorem.


\end{abstract}

\section{Introduction}

In the recent work \cite{BFIMZ2017} of
R. L. Bryant, P. Foulon, S. Ivanov, V. S. Matveev and W. Ziller,
the authors classified Finsler spheres with constant flag curvature $K\equiv 1$ according to the behavior of geodesics.
The Katok metric \cite{Ka1973}
provides the most important key model for their classification.
The celebrated Anosov Conjecture \cite{An1975}, claiming the minimal number of prime closed geodesics on a Finsler sphere $(S^n,F)$ is $2[\frac{n+1}{2}]$, was based on the discovery of
Katok metrics with only finite prime closed geodesics. There are many works using Morse theory and index theory to study the closed geodesics and Anosov Conjecture in Finsler geometry, assuming a pinch condition for the flag curvature, non-degenerating property for all closed geodesics, or using the speciality of low dimensions. See for example
\cite{BL2010}\cite{Du2016}\cite{DL2009}\cite{DLW2016}\cite{Ra1989}\cite{Wa2012}\cite{Wa2015}.
From the geometrical point of view, it was
much later that people noticed that Katok metrics are Randers metrics on spheres with
constant flag curvature \cite{Ra2004}. D. Bao, C. Robles and Z. Shen provided a complete classification for all Randers metrics with constant flag curvature\cite{BRS2004}.
The classification for the non-Randers case
is still widely open. R. L. Bryant provided many important
examples of Finsler spheres with $K\equiv 1$ \cite{Br1996}\cite{Br1997}\cite{Br2002}.

However, one of the most important technique in
\cite{BFIMZ2017} is from Lie theory. The authors considered
the antipodal map $\psi$
for a Finsler sphere with $K\equiv 1$ (see \cite{BFIMZ2017}\cite{Sh1997} or Section 2 for its definition).
It is a Clifford Wolf translation in the center of the isometry group $I(M,F)$.
When $\psi$ has an infinite order, after taking closure, it can be used to generate
a closed Abelian subgroup of isometries with a positive dimension.

For nonzero Killing vector fields on a Finsler sphere with $K\equiv 1$, we have the following totally geodesic technique.
The common zero point set of Killing vector fields, or more generally
the fixed point set of isometries, provide closed totally geodesic sub-manifolds. In particular, when the dimension of such a sub-manifold is one, it is a reversible geodesic, and when the dimension is even bigger, it is a Finsler sphere inheriting the curvature property and geodesic property from the ambient space. We can use this key observation to set up
an inductive argument, when studying the geodesics on $(S^n,F)$
with $n>2$ and $K\equiv 1$, and generalizing some results in
\cite{BFIMZ2017} to high dimensions.

For example, in \cite{Xu2018}, we have proved the following lower
bound estimate for the number of reversible prime closed geodesics in Finsler spheres with constant flag curvature.

\begin{theorem}\label{thm-1}
Let $(M,F)=(S^n,F)$ with $n>1$ be a Finsler sphere with $K\equiv 1$ and only finite prime closed geodesics. Then the number of geometrically distinct reversible closed geodesics is at least $\dim I(M,F)$.
\end{theorem}

Recall that a geodesic $c(t)$ with constant speed is called {\it reversible} if $c(-t)$ also provides a geodesic with constant
speed after a re-parametrization by the new arc length. Two geodesics are {\it geometrically distinct} iff they are different subsets.

The assumption of only finite prime closed geodesics imposes a strong restriction on $I_o(M,F)$, which can only be a torus. A lot of important examples are excluded, for example, the standard unit spheres and the homogeneous non-Riemannian Randers
spheres with $K\equiv 1$. So if we want more possibility for $I_o(M,F)$, the geodesic condition
could be replaced by the assumption that there exist only finite orbits of prime closed geodesics, or
Assumption (F) for simplicity. See Section 3 for its precise
definition and detailed discussion.

%

The main purpose of this paper is to prove the following
theorem.

\begin{theorem}\label{main-thm}
Let $(M,F)=(S^n,F)$ be a Finsler sphere satisfying $n>1$,
$K\equiv 1$ and Assumption (F). Denote $H$ the subgroup of
$G=I_o(M,F)$ preserving each closed geodesics, $H_o$ its identity component and $m=\dim H$.
Then there exist at least $m$ geometrically distinct orbits
$\mathcal{B}_i$'s of prime closed geodesics such that each union ${B}_i$ of geodesics in $\mathcal{B}_i$ is a totally geodesic
sub-manifold in $M$ with a non-trivial $H_o$-action.
\end{theorem}

When $(M,F)$ has only finite prime closed geodesics, then
Assumption (F) is satisfied, $H_o=G=I_o(M,F)$, and each
orbit of closed geodesics consists of only one closed geodesic. So Theorem \ref{main-thm} generalizes Theorem \ref{thm-1}.
It even slightly refines Theorem \ref{thm-1} by claiming
the totally geodesic ${B}_i$'s found have non-trivial
$H_o$-actions. So if the common zero point of $H_o$ has a
positive dimension, it provides one more totally geodesic ${B}_i$, which is either a reversible closed geodesic
which length is a rational multiple of $\pi$, or isometric to a standard unit sphere.

This paper is organized as following. In Section 2, we recall
some fundamental geometric properties of Finsler spheres with $K\equiv 1$, discussing their antipodal maps and totally geodesic
sub-manifolds. In Section 3,
we define Assumption (F), i.e. the assumption of
only finite prime closed geodesics.
In Section 4, we
introduce the subgroup $H$ of isometries which preserves
each closed geodesics.
In Section 5, we prove Theorem \ref{main-thm} by induction.
In Section 6, we discuss the Katok metrics, and
show that in some cases they provides examples for Theorem \ref{main-thm}, for which the estimate in Theorem \ref{main-thm} is sharp.

{\bf Acknowledgement.} The author would like to thank sincerely
Chern Institute of Mathematics, Nankai University, and Shaoqiang Deng for the hospitality during the preparation for
this paper. The author also thanks Yuri G. Nikonorov and Huagui Duan for helpful discussions.

\section{Preliminaries: from antipodal map to Killing vector field}

Let $(M,F)=(S^n,F)$ be a Finsler sphere satisfying the dimension $n>1$ and the flag curvature $K\equiv 1$.
Denote $G=I_o(M,F)$ the connected isometry group, i.e. the identity component of the isometry group $I(M,F)$ of $(M,F)$.


We briefly recall the definition of the exponential map \cite{BCS2000} and the antipodal map
$\psi$ \cite{BFIMZ2017} \cite{Sh1997} for $(M,F)$.

For any $x\in M$ and nonzero $y\in T_xM$, the {\it exponential map}
$\mathrm{Exp}_x:T_xM\rightarrow M$ is defined by $\mathrm{Exp}_x(y)=c(1)$ where
$c(t)$ is the constant speed geodesic with $c(0)=x$ and $\dot{c}(0)=y$. When $y=0\in T_xM$, we define $\mathrm{Exp}_x(0)=x$. Notice that $\mathrm{Exp}_x$
is $C^1$ at $y=0$ and $C^\infty$ elsewhere.

The discussion
for the Jacobi fields and conjugation points when $K\equiv 1$ indicates $\mathrm{Exp}_x$ maps the sphere
$$S^F_o(\pi)=\{y\in T_xM| F(y)=\pi\}\subset T_xM$$ to a single
point $x^*\in M$. The map from $x$ to $x^*$ is an isometry of $(M,F)$ in the center of $I(M,F)$ \cite{BFIMZ2017}. Further more, it is easy to see that $\psi$ is a Clifford Wolf
translation for the (possibly non-reversible) distance $d_F(\cdot,\cdot)$ defined by the Finsler metric $F$.
We will call it the {\it antipodal map} and always denote it
as $\psi$. It is a generalization for the antipodal map for
standard unit spheres but may not be an involution any more.

The above description immediately proves that any connected and simply connected Finsler manifold $(M,F)$ with $\dim M>1$ and $K\equiv 1$ is homeomorphic to
a sphere. A more careful discussion with the local charts shows that the homeomorphism in this statement can be refined to be
a diffeomorphism, and the argument is valid not only for $M$,
but also any closed connected totally geodesic sub-manifold $N$ with $\dim N>1$, i.e.
we have the following lemma (Lemma 3.2 in \cite{Xu2018}).

\begin{lemma} Let $(M,F)$ be a connected and simply connected Finsler manifold with $K\equiv 1$ and $N$ a closed connected totally geodesic sub-manifold with $\dim N>1$. Then both $M$ and $N$
are diffeomorphic to standard spheres, and $N$ is an imbedded
sub-manifold in $M$.
\end{lemma}

The fixed point set for a family of isometries in $I(M,F)$
is a closed, possibly disconnected, totally geodesic sub-manifold.
We have the following
lemma (Lemma 3.5 in \cite{Xu2018}), indicating the connectedness of $N$, when its dimension is positive.

\begin{lemma}\label{lemma-a1-3-5}
Let $(M,F)=(S^n,F)$ be a Finsler sphere with $n>1$ and
$K\equiv 1$, and $N$ the fixed point set of a family of
isometries of $(M,F)$. Then $N$ must satisfy one of
the following
\begin{description}
\item{\rm (1)} $N$ is a two-points $\psi$-orbit, i.e. $N=\{x',x''\}$ with $d_F(x',x'')=d_F(x'',x')=\pi$.
\item{\rm (2)} $N$ is a reversible closed geodesic.
\item{\rm (3)} $(N,F|_N)$ is a Finsler sphere with $\dim N>1$ and $K\equiv 1$.
\end{description}
\end{lemma}

The space of Killing vector fields can be viewed as the Lie algebra of $I(M,F)$. So the common zero set of a family of Killing vector
fields on $(M,F)$ is a special case of fixed point sets for
isometries.

In later discussions, we will need the following two lemmas
for Killing vector fields.

\begin{lemma}\label{lemma-killing-1}
Assume that $X$ is a Killing vector field
of the Finsler space $(M,F)$, $f(\cdot)=F(X(\cdot))$ and $f(x)>0$ at $x\in M$.
Then the integration curve of $X$ passing $x$ is a geodesic
iff $x$ is a critical point of $f(\cdot)$.
\end{lemma}

This is a direct corollary of Lemma 3.1 in \cite{DX2014}.

\begin{lemma} \label{lemma-killing-2}
Assume that $c=c(t)$ is a geodesic
of positive constant speed on the Finsler space $(M,F)$. Then restricted $c(t)$, any Killing vector field $X$ of $(M,F)$
satisfies
\begin{equation}\label{002}
\langle X(c(t)),\dot{c}(t)\rangle_{\dot{c}(t)}^F\equiv\mathrm{const},
\end{equation}
where $\langle u,v\rangle_y^F=g_{ij}u^iv^j$ for $u,v,y\in T_xM$ and $y\neq 0$ is the
inner product defined by the fundamental tensor.
\end{lemma}

\begin{proof}
Whenever the value of $X$ is linearly independent of $\dot{c}(t)$, we can prove (\ref{002}) by choosing
a special local chart, such that $c=c(t)$ can be presented
as $x^1=t$ and $x^i=0$ for $i>1$, and $X=\partial_{x^2}$.
Because $X$ is Killing vector field, $F(x,y)$ is independent
of $x^2$. The condition that $c=c(t)$ is a geodesic implies
that for the coefficients $\mathbf{G}^i$ of the geodesic spray,
we have
\begin{eqnarray*}
\mathbf{G}^i(c(t),\dot{c}(t))&=&
\frac{1}{4}g^{il}([F^2]_{x^my^l}y^m-[F^2]_{x^l})\\
&=&\frac{1}{4}g^{il}([F^2]_{x^1y^l}-[F^2]_{x^l})
=0.
\end{eqnarray*}
In particular, on the geodesic $c=c(t)$, we have
$$\langle X(c(t)),\dot{c}(t)\rangle_{\dot{c}(t)}^F
=\frac{1}{2}[F^2]_{x^1y^2}=\frac{1}{2}[F^2]_{x^2}=0,$$
which proves the lemma in this case.

When $X$ is tangent to $c=c(t)$ for $t$ in an interval $I$,
we can easily get (\ref{002}) for $t\in I$.

Summarizing this two cases and using the continuity, we have
proved (\ref{002}) along the whole geodesic $c=c(t)$.
\end{proof}

\section{Orbit of closed geodesics and Assumption (F)}

Now we define {\it Assumption (F)}, i.e. the condition that $(M,F)$ has only finite orbits of prime closed geodesics. In later discussion, we will always assume it to be
satisfied by $(M,F)$ unless otherwise specified.

The free loop
space $\Lambda M$ of all piecewise smooth path $c=c(t)$ with $t\in\mathbb{R}/\mathbb{Z}$ (sometimes we will simply denote it as $c$ or $\gamma$) admits the natural actions of
$\hat{G}=G\times T$ such that
$$((g,t')\cdot c)(t)=g\cdot c(t+t'),\quad\forall t.$$
So for each closed geodesic $\gamma$ of constant speed, we have an $\hat{G}$-orbit $\hat{G}\cdot \gamma$ of closed geodesics with the same speed. The geodesic $c(t)$ (with $t\in\mathbb{R}/\mathbb{Z}$) is {\it prime}, i.e.
$$\min\{t| t>0\mbox{ and }c(t)=c(0)\}=1,$$ iff all the closed geodesics in $\hat{G}\cdot c$ are prime.

\begin{definition}\label{def-1}
We say $(M,F)$ has only finite orbits of prime closed geodesics, or simply it satisfies Assumption (F), if all
the prime closed geodesics of positive constant speed
can be listed as a finite set of $\hat{G}$-orbits, $\mathcal{B}_i=\hat{G}\cdot \gamma_i$, $1\leq i\leq k$.
\end{definition}

In Definition \ref{def-1}, we can equivalently list all
the closed geodesics of constant speed $c(t)$ with $t\in\mathbb{R}/\mathbb{Z}$
as
$\mathcal{B}_i^j=\hat{G}\cdot\gamma_i^j$, $1\leq i\leq k$, $j\in\mathbb{N}$. The orbit $\mathcal{B}_i$ in Definition \ref{def-1} coincides with $\mathcal{B}_i^1$, for each $i$. The closed geodesics $\gamma_i^j$ is the one which
rotates $j$-times along the prime closed geodesic $\gamma_i$ in Definition \ref{def-1}, i.e. if $\gamma_i$ is presented as $c_i=c_i(t)$, then $\gamma_i^j$ is $c_{i,j}(t)=c_i(jt)$.

We denote ${B}_i$ the union of the geodesics in
$\mathcal{B}_i$ or $\mathcal{B}_i^j$ for any $j\in\mathbb{N}$. Then we call
$\mathcal{B}_i^j$ and $\mathcal{B}_{i'}^{j'}$ {\it geometrically distinct}
(or {\it geometrically the same}), if ${B}_i$
and ${B}_{i'}$ are different
subsets (or the same subsets, respectively) of $M$.

The Assumption (F) for the ambient space can be inherited by some
totally geodesic sub-manifolds, i.e. we have the following lemma.

\begin{lemma}\label{lemma-0}
Let $(M,F)$ be any closed compact Finsler manifold satisfying Assumption (F),
$\phi_\alpha$ with $\alpha\in\mathcal{A}$ a family of isometries in the center of $I(M,F)$, and
$N$ the fixed point set for all $\phi_\alpha$'s. Then
each orbit of prime closed geodesic for $(N,F|_{N})$
is also an orbit of prime closed geodesic for $(M,F)$. In particular, $(N,F|_{N})$ also satisfies Assumption (F).
\end{lemma}

\begin{proof}
The fixed point set $N$ for the isometries $\phi_\alpha$ with $\alpha\in\mathcal{A}$
is a closed (possibly disconnected) totally geodesic
sub-manifold of $(M,F)$. Because each $\phi_\alpha$ commutes with all isometries of $(M,F)$, the fixed point set $N$ for all $\phi_\alpha$'s is preserved by
the action of $G=I_o(M,F)$. The restriction of $G$-action to $N$ defines isometries in $G'=I_o(N,F|_{N})$. Denote  $\hat{G}'=G'\times T$. Then for each prime closed geodesic $\gamma$ in $N$, Assumption (F) implies that
$\hat{G}'\cdot \gamma$ is a disjoint finite union of $\hat{G}$-orbits. Both $\hat{G}'$-orbits and $\hat{G}'$-orbits are compact and connected, so we get $\hat{G}'\cdot\gamma=\hat{G}\cdot\gamma$, which proves the first claim.
The second claim follows immediately.
\end{proof}

The effect of Assumption (F) can be seen from the behavior of
the antipodal map $\psi$. For example, when $\psi$ has a finite order $k$, i.e. there exists a positive integer $k$, such that
$$\psi^k=\mathrm{id}, \mbox{ and }\psi^i\neq\mathrm{id}
\mbox{ when } 1\leq i<k,$$
we have the following lemma.

\begin{lemma}\label{lemma-1}
Let $(M,F)=(S^n,F)$ be a Finsler sphere satisfying $n>1$, $K\equiv 1$ and Assumption (F). Assume that the antipodal
map $\psi$ has a finite order $k$, then $F$ must be
the Riemannian metric for a standard unit sphere.
\end{lemma}

\begin{proof}
Because $\psi$ is a Clifford Wolf translation, and it has
a finite order $k$, each geodesic of $(M,F)$ is closed, and
each prime closed geodesic admits a suitable multiple such that
the length of the resulting closed geodesic is $k\pi$.
By Assumption (F), the subset $B\subset\Lambda M$ of
all closed geodesics with the length $k\pi$ can be listed as the disjoint union of
$\mathcal{B}_i^{n_i}=\hat{G}\cdot\gamma_i^{n_i}$, $1\leq i\leq k$, where each $\gamma_i$ is a prime closed geodesic. Obviously $B$ is connected and each $\mathcal{B}_i^{n_i}$ is compact, so
we must have $k=1$.

Then we prove $(M,F)$ is $G$-homogeneous. Assume conversely
that it is not, we consider a unit speed geodesic $c(t)$, and
the $G$-orbit $N$ passing $c(0)$, such that
\begin{equation}\label{000}
\langle \dot{c}(0),T_{c(0)}N\rangle^F_{\dot{c}(0)}=0.
\end{equation}
Then by Lemma \ref{lemma-killing-2}, for any Killing vector field $X\in\mathfrak{g}$, we have
$$\langle\dot{c}(t),X(c(t))\rangle_{\dot{c}(t)}^F\equiv 0,$$
i.e. $c(t)$ meets each $G$-orbit orthogonally in the sense of (\ref{000}). This property is
 preserved by $\hat{G}$-actions. So its $\hat{G}$-orbit can not exhaust all
the geodesics, for example, those which does not satisfy (\ref{000}).
This is a contradiction to our previous observation that $(M,F)$ can only have one orbit of prime closed geodesics, and it proves that $(M,F)$ is homogeneous Finsler sphere.

Finally, we prove $(M,F)$ is a standard unit sphere. Because $(M,F)$ is a homogeneous
Finsler space, it has at least one homogeneous geodesic $c(t)=\exp(tX)\cdot o$, in which $o\in M$ and $X\in \mathfrak{g}=
\mathrm{Lie}(G)$ \cite{HY2017}. Our previous observation that all geodesics belong to a single $\hat{G}$-orbit implies all geodesics are homogeneous. So for
any $x\in M$ and any two $F$-unit tangent vectors $y_1$ and $y_2$ in $T_xM$, we have two unit speed geodesics $c_1(t)$ and
$c_2(t)$ such that $c_1(0)=c_2(0)=x$ and $\dot{c}_i(0)=y_i$.
Both geodesics belong to the same $\hat{G}$-orbit, so we can find $g_1\in G$ such that $(g_1\cdot c_1)(t)\equiv c_2(t+t_0)$ for some fixed $t_0$. Because the geodesic $c_2(t)$ is homogeneous, we can find another $g_2\in G$ such that
$(g_2\cdot c_2)(t)=c_2(t-t_0)$. Then we have
$$(g_2g_1\cdot c_1)(t)=(g_2\cdot c_2)(t+t_0)=c_2(t),\quad\forall t.$$
So the isotropy action for $(M,F)$ is transitive at each point. The only homogeneous spheres satisfying this property are Riemannian
spheres of constant curvature.

This ends the proof of the lemma.
\end{proof}

Using Lemma \ref{lemma-0} and Lemma \ref{lemma-1}, we can
generalize Lemma 3.6 in \cite{Xu2018} to the following.

\begin{lemma}\label{lemma-2}
Let $(M,F)=(S^n,F)$ be a Finsler sphere satisfying $n>1$, $K\equiv 1$ and Assumption (F). Then the union $N$ of all the finite orbits of $\psi$ in $M$ must be one of the following:
\begin{description}
\item{\rm (1)} A two-points $\psi$-orbit.
\item{\rm (2)} A closed reversible geodesic which length is rational multiple of $\pi$.
\item{\rm (3)} A Riemannian sphere of constant curvature isometrically imbedded in $(M,F)$ as a totally geodesic
    sub-manifold. In this case we have $k=2$.
\end{description}
\end{lemma}

\begin{proof}
By the same argument as in the proof of Lemma 3.6 in \cite{Xu2018}, we can proof $N$ is the fixed point set
of $\psi^k$ for some integer $k$, hence it is totally
geodesic in $(M,F)$.
When $\dim N=0$ or $1$, we get the cases (1) and (2)
respectively. The difference appears when $\dim N>1$,
which may happen with the finite orbit of prime closed
geodesics condition. When $\dim N>1$, by Lemma \ref{lemma-a1-3-5},
$(N,F|_N)$ is a Finsler sphere satisfying $K\equiv 1$. By Lemma
\ref{lemma-0}, $(N,F|_N)$ also satisfies Assumption (F).
Then Lemma \ref{lemma-1} provides the case (2) in the lemma.
\end{proof}

The cases (2) and (3) cover all the possibilities for the
$\hat{G}$-orbit of a prime closed geodesic $\gamma$ such
that the length of $\gamma$ is a rational multiple of $\pi$.

Next, we consider the $\hat{G}$-orbit of a prime closed geodesic $\gamma$ such that the length of $\gamma$ is an
irrational multiple of $\pi$.

When the length of $\gamma$ is an irrational multiple of $\pi$,
any $\psi$-orbit in $\gamma$ is dense. Following this observation, we can easily prove the following lemma.

\begin{lemma}\label{lemma-3}
Let $(M,F)=(S^n,F)$ be a Finsler sphere satisfying $n>1$, $K\equiv 1$
and Assumption (F). Then two geometrically distinct closed geodesics can intersect iff
they are intersecting geodesics in the totally geodesic sub-manifold in $(M,F)$ which is isometric to a unit sphere, i.e. the case (3) in Lemma \ref{lemma-2}.
\end{lemma}

\begin{proof}
Lemma \ref{lemma-2} indicates that any two geometrically distinct closed geodesics $\gamma_1$ and $\gamma_2$ must satisfy one of the following. Either both lengths
are $2\pi$ or one of them, for example $\gamma_1$, has a length which is an irrational multiple of $\pi$.
In the first case, they are contained in a totally geodesic
sub-manifold of $(M,F)$ which is isometric to a unit
sphere. In the second case, the intersection of the two geodesics contains a $\psi$-orbit, which is dense in $\gamma_1$. Both geodesics are closed, so does their intersection. So as subsets of $M$, we have $\gamma_1\subset\gamma_2$ and furthermore the equality must happen because $\gamma_2$ is a closed connected curve. This is the contradiction ending the proof of the lemma.
\end{proof}

Using above lemmas, we can provide more explicit description for the orbits of prime closed geodesics by the following lemma.
\begin{lemma}\label{lemma-3-5}
Assume $(M,F)=(S^n,F)$ is a Finsler sphere satisfying
$n>1$, $K\equiv 1$, Assumption (F), and that it is not
the standard unit sphere. Then we have the following:
\begin{description}
\item{\rm (1)} There exists closed geodesics
which lengths are irrational multiples of $\pi$.
\item{\rm (2)} For the orbit of prime closed geodesics $\mathcal{B}_i=\hat{G}\cdot\gamma_i$ such that the length of $\gamma_i$ is an irrational multiple of $\pi$, the
    corresponding ${B}_i$ is an orbit for the
    action of $G=I_o(M,F)$.
\item{\rm (3)} Two different orbits of prime closed geodesics, $\mathcal{B}_i$ and $\mathcal{B}_j$, are geometrically distinct
    iff ${B}_i$ and ${B}_j$ do not intersect.
\item{\rm (4)} Two different orbits of prime closed geodesics
$\mathcal{B}_i$ and $\mathcal{B}_j$ are geometrically the same iff we can
find $\gamma_i\in \mathcal{B}_i$ and $\gamma_j\in \mathcal{B}_j$ such that
$\gamma_i$ and $\gamma_j$ are the same curve with different
directions.
\end{description}
\end{lemma}

\begin{proof}
By Lemma \ref{lemma-1} and the assumption that $(M,F)$
is not the standard unit sphere, the antipodal map $\psi$
generates an infinite subgroup in $I(M,F)$, which
closure is a subgroup in the center of $I(M,F)$, corresponding
to an Abelian subalgebra $\mathfrak{c}'\subset\mathfrak{c}(\mathfrak{g})$ with $\dim\mathfrak{c}'>0$. We can find a nonzero Killing vector field $X$ from $\mathfrak{c}'$ which generates an $S^1$. Obviously, $X$ is tangent to each closed geodesic. The restriction of $X$ to each closed geodesic
which length is a rational multiple of $\pi$ is zero.

To prove (1), we only need to consider a maximum point $x$ of
$f(\cdot)=F(X(\cdot))$. By Lemma \ref{lemma-killing-1}, the integration curve $\gamma$ of $X$ passing $x$ is a geodesic, restricted
to which $X$ is nonzero. Because $X$ generates an $S^1$,
$\gamma$ is closed. So it is a closed geodesic which length
is an irrational multiple of $\pi$.

To prove (2), we consider a prime closed geodesic $\gamma_i$
which length is an irrational multiple of $\pi$. Because the
restriction of $X$ to $\gamma_i$ is a nonzero tangent vector
field, $\gamma_i$ is a homogeneous geodesic. In its $\hat{G}=G\times T$-orbit, The $T$-action on $\gamma_i$ can
be replaced by the actions of $\exp(tX)\in G$. So the union
${B}_i$ for the geodesics in $\mathcal{B}_i$ is a $G$-orbit.

The statements (3) and (4) follows immediately Lemma \ref{lemma-3}.
\end{proof}

Finally, we end this section with the following corollary.

\begin{corollary}\label{cor-1}
Assume $(M,F)=(S^n,F)$ is a homogeneous Finsler sphere satisfying $n>1$, $K\equiv 1$ and Assumption (F). Then
all closed geodesics are reversible. Furthermore,
one of the following two cases must happen:
\begin{description}
\item{\rm (1)} $(M,F)$ is a standard unit sphere. It has exactly one orbit of prime closed geodesics and all geodesics are closed.
\item{\rm (2)} The dimension $n$ of $(M,F)$ is odd. There exists exactly two orbits of prime closed
geodesics $\hat{G}\cdot\gamma_1$ and $\hat{G}\cdot\gamma_2$,
in which $\gamma_1$ and $\gamma_2$ are the same curve with different directions.
\end{description}
\end{corollary}

Almost all the statement follows immediately Lemma \ref{lemma-3-5}. For the case (2) in Corollary \ref{cor-1},
the dimension $n$ of $(M,F)$ is odd because $G=I_o(M,F)$
acts transitively and effectively on $M$ and it has a center of positive dimension. Until now, the only know examples for this case are non-Riemannian Randers spheres $(S^{2n'-1},F)$ in which $n'>1$,
and $F$ is defined by the navigation datum $(h,W)$, where $h$
is the standard unit sphere metric, and $W$ is Killing
vector field such that its length function is constantly an
irrational number.

\section{Isometries preserving each closed geodesic}

Assume $(M,F)=(S^n,F)$ is a Finsler sphere satisfying $n>1$,
$K\equiv 1$, and Assumption (F). Let $\psi$ be its antipodal
map. By Lemma \ref{lemma-1}, the case that $\psi$ has a finite order
is easy, so in the following discussion we assume that
$\psi$ has an infinite order.

Denote $H$ the subgroup of $G=I(M,F)$ which preserves
each closed geodesic, $H_o$ its identity component, and $\mathfrak{h}$ its Lie algebra.
The group $H$ is intersection of
$$G_\gamma=\{g\in G|(g\cdot \gamma)(t)\equiv\gamma(t+t_0) \mbox{ for some }t_0\}$$
for all closed geodesics $\gamma$. Each $G_\gamma$ is
a closed subgroup of $G$. Then so does $H$.

It should be remarked that the claim that $G_\gamma$ is
a closed subgroup of $G$ is
an easy fact in this case because
$\gamma$ is closed. In the recent work \cite{BN2018}, it has been proved that $G_\gamma$ is still a Lie group when $\gamma$ is not closed.

Obviously the antipodal map $\psi$ belongs to $H$. Because
$\psi$ has an infinite order, then after taking closure, it generates an Abelian subgroup of positive dimension, i.e.
we have $\dim H>0$. The following lemma claims that $H_o$ commutes
with all the $G$-actions.

\begin{lemma}\label{lemma-4} The subgroup
$H_o$ is a closed subgroup in the center of $G=I_o(M,F)$.
\end{lemma}

\begin{proof} The previous observations have already proved that $H_o$ is a closed subgroup of $G$.
Because $G$ is a compact Lie group, to prove this lemma we only need to prove $\mathfrak{h}=\mathrm{Lie}(G)$ is
an Abelian ideal of $\mathfrak{g}$.

The Lie algebra
$\mathfrak{h}=\mathrm{Lie}(H)$ consists of all the
Killing vector fields $X$ which is tangent to each closed geodesic. Because the action of $G$ permutes the closed geodesics in each orbit of prime closed geodesics, any Killing vector field of the form $\mathrm{Ad}(g)X$ for $g\in G$ and $X\in\mathfrak{h}$ is also tangent to each closed geodesic. So
conjugations of $G$ preserves $\mathfrak{h}$,
i.e. $\mathfrak{h}$
is an ideal of $\mathfrak{g}$.

Then we prove $\mathfrak{h}$ is
Abelian by contradiction. Assume conversely that $\mathfrak{h}$ is not Abelian, then we can find a nonzero vector $X$ from the compact semi-simple Lie algebra $[\mathfrak{h},\mathfrak{h}]$ which generates an $S^1$-subgroup. The Killing
vector field on $(M,F)$ induced by $X$ has trivial restriction on each closed geodesic. By Lemma \ref{lemma-killing-1}, the integration curve of $X$ passing the maximum point of $f(\cdot)=F(X(\cdot))$ is
a closed geodesic. This is a contradiction which ends the proof
of this lemma.
\end{proof}

A direct consequence of Lemma \ref{lemma-4} is the following
lemma.

\begin{lemma}\label{lemma-5}
For any Killing vector field $X\in\mathfrak{h}$ and any
orbit $\mathcal{B}_i$ of the prime closed geodesic $c=c(t)$, their exists a constant
$\rho_{X,i}\in\mathbb{R}$ such that
\begin{equation}\label{001}
X|_{c(t)}\equiv\rho_{X,i}\dot{c}(t),\quad\forall c\in \mathcal{B}_i.
\end{equation}
\end{lemma}

In particular, if a Killing vector field $X\in\mathfrak{h}$ vanishes at some point $x\in\mathcal{B}_i$, iff $\rho_{X,i}=0$, i.e.
$X$ vanishes identically on ${B}_i$.

The last ingredient for the proof of Theorem \ref{main-thm}
is the following lemma.

\begin{lemma}\label{lemma-6}
Let $(M,F)=(S^n,F)$ be a Finsler sphere satisfying $n>1$, $K\equiv 1$ and Assumption (F). Then we have the following:
\begin{description}
\item{\rm (1)} For any nonzero Killing vector field $X\in\mathfrak{h}$ which generates an $S^1$, there exists some orbit $\mathcal{B}_i$ of prime closed geodesics, such that  $\rho_{X,i}>0$.
\item{\rm (2)} Any Killing vector field $X\in\mathfrak{h}$ vanishing on
    all closed geodesics must be a zero vector field.
\item{\rm (3)}
The common zero set of all Killing vector fields in
$\mathfrak{h}$ must be the fixed point set of $\psi^k$ for some integer $k$. To be more precise, it is empty, a two-points $\psi$-orbit, some ${B}_i$ which is
a reversible closed geodesic which lengths for both directions
 are rational multiples of $\pi$, or a totally geodesic sub-manifold isometric to a standard unit sphere.
 \end{description}
\end{lemma}

\begin{proof}
(1) We consider the maximum point
$x$ for the function $f(\cdot)=F(X(\cdot))$. By Lemma
\ref{lemma-killing-1}, the integration curve of $X$ passing
$x$ provide a prime closed geodesic $\gamma$, for which
we have $X(c(t))\equiv \rho_{X,\gamma}\dot{c}(t)$ with $\rho_{X,\gamma}>0$.

(2) We assume conversely that there exists a non zero Killing vector field on $(M,F)$ such that it vanishes
on all closed geodesics.
Let $\mathfrak{k}$ be the space of all such Killing vector
fields. It is a subalgebra of $\mathfrak{h}$ corresponding
to a sub-torus in $H_o$. We can find a nonzero Killing vector
field $X$ from $\mathfrak{k}$ which generates an $S^1$. The argument for (1) indicates $X$ is not vanishing on some closed geodesic, which is the contradiction.

(3) Let $N$ be the fixed point set of $H_o$, and assume $N$ is not empty. By Lemma \ref{lemma-a1-3-5}, $N$ must be a
two-points $\psi$-orbit, a reversible closed geodesic, or
a Finsler sphere with $\dim N>1$, $K\equiv 1$ isometrically imbedded in $(M,F)$.

Obviously the action of $\psi$ preserves
$N$, i.e. $N$ consists of $\psi$-orbits.
Because $H$ is compact, $H/H_o$ is finite. We also have $\psi\in H$, and thus
each $\psi$-orbit in $N$ is finite. So when $\dim N=1$, the
lengths of $N$ for both directions are rational multiples of
$\pi$.

When $\dim N>1$, we see $(N,F|_N)$ satisfies Assumption (F)
by Lemma \ref{lemma-0}. Then Lemma \ref{lemma-1} tells us
that $(N,F|_N)$ is a standard unit sphere.
\end{proof}

\section{Proof of Theorem \ref{main-thm}}

Now we are ready to prove Theorem \ref{main-thm}, which applies a similar inductive argument
as that for Theorem 1.2 in \cite{Xu2018}.

When $\psi$ has a finite order, then by Lemma \ref{lemma-1},
$(M,F)$ is the standard unit sphere. Obviously Theorem \ref{main-thm} is valid in this case. So in the following discussion, we assume $\psi$ has an infinite order, and
thus we have $m=\dim H>0$.

We will prove Theorem \ref{main-thm} by an induction for
$n=\dim M$.

When $n=2$ and the antipodal map $\psi$ has an infinite order, $H_o$ coincides with $G=I_o(M,F)=S^1$. In \cite{BFIMZ2017}, it has been proved that geometrically there exists exactly one reversible closed geodesic $\gamma$ with a non-trivial $H_o$-action. So Theorem \ref{main-thm} is valid in this case, and the estimate
is sharp.

Now we assume Theorem \ref{main-thm} is valid when $n<l$ with $l>3$ (the inductive assumption) and we will prove the theorem when $n=l$.

Firstly, we prove

{\bf Claim 1}: When $\dim H=1$, there exists at least one totally geodesic
${B}_i$ with a non-trivial $H_o$-action.

Let $X$ be any nonzero Killing vector field from $\mathfrak{h}=\mathrm{Lie}(H)$. We list
all the $\hat{G}$-orbits of prime closed geodesics
as $\mathcal{B}_i$ with $1\leq i\leq k$,
such that when $1\leq i\leq k'$
the coefficient $\rho_{X,i}$ in (\ref{001})
is positive. Notice that by Lemma \ref{lemma-6} (1), we have $k'>0$.

If the antipodal map $\psi$ is not contained in $H_o$,
we can find an isometry of $(M,F)$ which is
of the form $\phi=\psi\exp (t'X)$ such that its fixed point set
contains ${B}_1$. By Lemma \ref{lemma-a1-3-5} (or see Lemma 3.5 in \cite{Xu2018}),
the fixed point set $N$ of $\phi$ is a closed connected totally geodesic sub-manifold. It must have a positive co-dimension in $M$ because $\phi\notin H_o$. When $\dim N=1$, it is
a reversible closed geodesic. When $\dim N>1$, by Lemma \ref{lemma-0} and the totally geodesic property, $(N,F|_N)$ is
a Finsler sphere satisfying $K\equiv1$ and Assumption (F). Using the inductive
assumption, we can find some orbit of prime closed geodesic, $\mathcal{B}_i=\hat{G}'\cdot\gamma_i=\hat{G}\cdot\gamma_i$, where $\hat{G}'=G'\times T$ and $G'=I_o(N,F|_N)$, such that the corresponding $\mathcal{B}'_i$, is totally geodesic in $(N,F|_N)$ as well as in $(M,F)$. The $H_o$-action on ${B}_i$ is non-trivial because
$$\exp(t'X)|_{{B}_i}=\psi^{-1}\phi|_{{B}_i}
=\psi^{-1}|_{{B}_i},$$
and $\psi$ has no fixed point on any closed geodesic.

To summarize, this proves Claim 1 when $\psi\notin H_o$.

To continue the proof of Claim 1, we may assume $\psi\in H_o$.
In this case, we can prove the zero set of $X$ is empty as following. Assume conversely that the zero set of $X$ is
not empty, by Lemma \ref{lemma-6}, it is a two-points $\psi$-orbit, a reversible closed geodesic, or a connected totally geodesic standard unit sphere. For each possibility, $\psi$ can not
be generated by $X$, which is a contradiction to the assumption
$\psi\in H_o$. This fact implies that $f(\cdot)=F(X(\cdot))$
is a smooth function on $M$. By Lemma \ref{lemma-killing-1}, the critical point set of $f(\cdot)$
consists of exactly all ${B}_i$'s with $1\leq i\leq k'$. Meanwhile, we see the $H_o$-action on each closed geodesic
is non-trivial.

We take a prime closed geodesic $c_i(t)$ with $t\in \mathbb{R}/\mathbb{Z}$ from $\mathcal{B}_i$ for $1\leq i\leq k'$,
then $X|_{c_i}=\rho_{X,i}\dot{c}_i$ with $\rho_{X,i}>0$.
Because $H_o=S^1$, we can find some $t'>0$ such that
$\exp(t'X)=\mathrm{id}$, then we have
$$n_i=t'\rho_{X,i}\in\mathbb{N},\quad \forall 1\leq i\leq k'.$$
We may re-order these $c_i$'s such that $$n_1\leq n_2\leq \cdots\leq n_{k'}.$$
There are two possibilities, all $n_i$'s are not all the same, or
all $n_i$'s are all the same.

Assume all $n_i$'s are not all the same, i.e. $n_1<n_{k'}$. The fixed
point set $N$ of the isometry $\phi=\exp((t'/n_{k'})X)\in H_o$ contains ${B}_{k'}$ but not ${B}_1$.
It is either a reversible closed geodesic, or a Finsler sphere satisfying $1<\dim N<\dim M$, $K\equiv1$ and Assumption (F). Applying the inductive assumption and Lemma \ref{lemma-0},
we can find a totally geodesic ${B}_i$ for $(N,F|_N)$,
as well as for $(M,F)$.

Assume all $n_i$'s are all the same, then all $\rho_{X,i}$'s are all the same as well. We may choose a suitable $t'$ such that $n_i=1$ for $1\leq i\leq k'$. There exists $t''\in(0,1)$ such that $\psi(c_i(0))=c_i(t'')$, i.e. $d_F(c_i(0),c_i(t''))=\pi$, for $1\leq i\leq k'$. Then we have
$$F(X|_{c_1})=F(X|_{c_2})=\cdots=F(X|_{c_{k'}}).$$
The function $f(\cdot)=F(X(\cdot))$ takes the same value on its critical
point set, so it is a constant function. By Lemma \ref{lemma-killing-1},
all integration curves of $X$ are closed geodesics, which belongs to one $\hat{G}$-orbit. By Corollary \ref{cor-1}, $(M,F)$ is a non-Riemannian homogeneous Finsler sphere with $K\equiv 1$ and exactly two $\hat{G}$-orbits of prime closed
geodesics, $\mathcal{B}_1=\hat{G}\cdot\gamma_1$ and $\mathcal{B}_2=\hat{G}\cdot\gamma_2$ such that $\gamma_1$ and $\gamma_2$ are the same curve with different directions.

This ends the proof of Claim 1, i.e. Theorem \ref{main-thm} is
valid when $m=\dim H=1$.

Next we prove Theorem \ref{main-thm} assuming
$m=\dim H>1$. We claim

{\bf Claim 2}: There exists at least $m-1$ geometrically distinct orbits $\mathcal{B}_i$
such that the each ${B}_i$ is a totally geodesic
sub-manifold with a non-trivial $H_o$-action.

Let $\mathcal{B}_i$ with $1\leq i\leq k'$ be all the geometrically distinct
$\hat{G}$-orbits of prime closed geodesics such that the
$H_o$-action on each ${B}_i$ is not trivial. Let $\mathfrak{h}_i$
be the co-dimension one subalgebra of $\mathfrak{h}$ which
restriction to ${B}_i$ is zero. By Lemma \ref{lemma-6}, the intersection $\cap_{i=1}^{k'}{\mathfrak{h}_i}=0$, from which we see
that $m\leq k'$. We may re-order the orbits $\mathcal{B}_i$'s such that
$\cap_{i=1}^m\mathfrak{h}_i=0$. Take a nonzero Killing vector field
$X\in\cap_{i=1}^{m-1}\mathfrak{h}_i$.
Then the zero set $N$ of $X$ is a closed connected totally geodesic submanifold in $M$, containing ${B}_i$ for $1\leq i\leq m-1$ but
not ${B}_m$. Let $H'$ be the subgroup of $I_o(N,F|_N)$
preserving all closed geodesics in $N$, and $\mathfrak{h}'$ its
Lie algebra.
The restriction from $M$ to $N$ defines
a linear map from $\mathfrak{h}$ to
$\mathfrak{h}'$ which kernel is spanned by $X$, so
$\dim H'\geq m-1$.

If $\dim N=1$, then $m=2$, $H_o$
has no fixed point, and $N$ itself provides the totally geodesic ${B}_i$ wanted by Claim 2.

If $\dim N>1$, we can use the inductive assumption to find $m-1$ geometrically distinct orbits $\mathcal{B}_i$
of prime closed geodesics for $(N,F|_N)$, as well as for $(M,F)$ by Lemma \ref{lemma-0}, such that the corresponding
${B}_i$'s are totally geodesic sub-manifolds, with non-trivial $H'_o$-actions. Claim 2 is proved when each of these ${B}_i$'s also has a non-trivial $H_o$-action.

But it is possible that there is some ${B}_i$ in $N$ on
which
the $H'_o$-action is non-trivial but the $H_o$-action is trivial. If it happens, this ${B}_i$ is unique, and
we must have $\dim H'>m-1$. So in this case, we can use the
inductive assumption to find $m$ geometrically distinct orbits
of prime closed geodesics. At least $m-1$ geometrically distinct totally geodesic ${B}_i$'s in $N$ have non-trivial $H_o$-actions.

This proves Claim 2.

To finish the proof of Theorem \ref{main-thm} when $n=k$,
We only need to find one more totally geodesic ${B}_i$ with a non-trivial $H_o$-action.

We may re-order the orbits $\mathcal{B}_i$'s such that the
first $m-1$ ones are those provided by Claim 2, and
$\cap_{i=1}^m\mathfrak{h}_i=0$. The
nonzero Killing vector field $X$ from $\cap_{i=1}^{m-1}\mathfrak{h}_{i}$ vanishes on ${B}_i$ with $1\leq i\leq m-1$, but not on ${B}_m$. We can find an isometry of the form
$\phi=\psi\exp (t'X)$ such that it fixes each point of ${B}_m$. On the other hand,
the fixed point set $N$
of $\phi$ does not contain each $\mathcal{B}_i$ for $1\leq i\leq m-1$.

The $H_o$-action on each closed geodesic in $N$ is non-trivial. Assume conversely that there is a closed geodesic in $N$ with a trivial $H_o$-action. Then
the restriction of $\psi$ to this geodesic coincides
with that of $\phi$, fixing each point of this geodesic. This
is not true because $\psi$ has no fixed points.

If $\dim N=1$ it is
a reversible closed geodesic, which is the extra ${B}_i$ we want. If $\dim N>1$ it is a Finsler
sphere satisfying $K\equiv 1$ and Assumption (F), isometrically imbedded in $(M,F)$ as a totally geodesic sub-manifold. In this
situation we use the inductive assumption one more time, which
provides one more totally geodesic ${B}_i$.

Summarizing above discussion, we have proved
Theorem \ref{main-thm} when $n=l$.

This ends the proof of Theorem \ref{main-thm} by induction.

\section{The example from Katok metrics}

We conclude this paper by the examples from Katok metrics
for which the estimate in Theorem \ref{main-thm} is sharp.

Let $(M,h)=(S^n,h)$ be a standard unit sphere with $n>1$,
$W$ a Killing vector field on $(M,h)$ such that
$h(W,W)<1$ everywhere.

Then the navigation process defines a Randers metric
\begin{equation*}
F(y)=\frac{\sqrt{
\lambda h(y,y)+h(W,y)^2}}{\lambda}-\frac{h(W,y)}{\lambda}
\end{equation*}
on $M$, in which $\lambda=1-h(W,W)$ is positive everywhere.

By the work of D. Bao, C. Robles and Z. Shen \cite{BRS2004},
this construction provides all the Randers spheres with $K\equiv 1$. The behavior of the geodesics on $(M,F)$
is determined by the choice of $W$.

We can find suitable coordinates $x=(\mathbf{x}_0,\mathbf{z}_1,\ldots,\mathbf{z}_k)$
for $x\in\mathbb{R}^{n+1}$, where
$$\mathbf{x}_0=(x_{0,1},\ldots,x_{0,n_0})\in\mathbb{R}^{n_0}
\mbox{ and }
\mathbf{z}_i=(z_{i,1},\ldots,z_{i,n_i})\in\mathbb{C}^{n_i},$$
satisfy the following:
\begin{description}
\item{\rm (A1)} We permit $n_0=0$ and in this case $\mathbf{x}_0$ is always 0. All other $n_i$'s are positive.
\item{\rm (A2)} $(M,h)$ is naturally identified as the unit sphere $S^n(1)$ defined by
$$|\mathbf{x}_0|^2+|\mathbf{z}_1|^2+\cdots+|\mathbf{z}|^2=1$$
in $\mathbb{R}^{n+1}=\mathbb{R}^{n_0}\oplus
\mathbb{C}^{n_1}\oplus\cdots\oplus\mathbb{C}^{n_k}$ with the
standard product Euclidean metric.
\item{\rm (A3)}
$W$ can be presented as
\begin{equation}\label{003}
W(\mathbf{x}_0,\mathbf{z}_0,\ldots,\mathbf{z}_k)
=(0,\sqrt{-1}\lambda_1\mathbf{z}_1,
\ldots,\sqrt{-1}\lambda_k\mathbf{z}_k),
\end{equation}
such that $0<\lambda_1<\lambda_2<\ldots<\lambda_k<1$.
\end{description}

We further require one of the following is satisfied:
\begin{description}
\item{\rm (A4)} All $\lambda_i$'s are irrational numbers. For
any $1\leq i<j\leq k$,
$1$, $\lambda_i$ and $\lambda_j$ are linearly independent
over $\mathbb{Q}$.
\item{\rm (A5)} All $\lambda_i$'s are irrational numbers except one, $n_0$=0 and $n_i=1$ if $\lambda_i\in\mathbb{Q}$.
    If $\lambda_i$ and $\lambda_j$ are irrational numbers, $1$, $\lambda_i$ and $\lambda_j$
    are linearly independent.
\end{description}

Then we have
\begin{lemma} \label{lemma-7}
For the Randers sphere $(M,F)$ described above, satisfying (A1)-(A3) and one of (A4) and (A5), any closed geodesic on $(M,F)$ must be contained in
$$\mathbf{z}_1=\cdots=\mathbf{z}_k=0$$
or
$$\mathbf{x}_0=0\mbox{ and }\mathbf{z}_j=0\mbox{ when }j\neq i,$$
for some $i$, $1\leq i\leq k$.
\end{lemma}

\begin{proof}
Using (\ref{003}), we can present the antipodal map as
$$\psi(\mathbf{x}_0,\mathbf{z}_1,\ldots,\mathbf{z}_k)
=(\mathbf{x}_0,-e^{\sqrt{-1}\pi\lambda_1}\mathbf{z}_1,\ldots,
-e^{\sqrt{-1}\pi\lambda_k}\mathbf{z}_k).$$
It is easy to check that finite $\psi$-orbits only appear
in the situation that only $\mathbf{x}_0$ is nonzero or only $\mathbf{z}_i$
with $\lambda_i\in\mathbb{Q}$ is nonzero.

Let $x=(\mathbf{x}_0,\mathbf{z}_1,\ldots,\mathbf{z}_k)$
be a point on the closed geodesic $\gamma$. We only need to
prove that only one of $\mathbf{x}_0$ and $\mathbf{z}_i$'s
can be nonzero. Assume conversely this is not true. Then
the length of $\gamma$ can not be a rational multiple of
$\pi$ (i.e. consists of finite $\psi$-orbits), so the $\psi$-orbit of $x$ is a dense subset in $\gamma$. There are three cases we need to consider.

In the first case,
$\lambda_i$ and $\lambda_j$ are irrational numbers,
$\mathbf{z}_i\neq 0$, and $\mathbf{z}_j\neq 0$.
Then the condition that 1, $\lambda_i$ and $\lambda_j$
are linearly independent implies that the projection to
the $\mathbf{z}_i$- and $\mathbf{z}_j$-factors maps the closed curve $\gamma$
onto a two dimensional torus, which is a contradiction.

In the second case, $\lambda_i$ is rational, $\lambda_j$ is not,
$\mathbf{z}_i\neq 0$ and $\mathbf{z}_j\neq 0$.
Then the projection to the $\mathbf{z}_i$-factor maps $\gamma$ to a finite set with at least two points. This is impossible
because $\gamma$ is connected.

In the third case, $\mathbf{x}_0\neq 0$ and $\mathbf{z}_i\neq 0$. Then the projection to the $\mathbf{x}_0$-factor maps
$\gamma$ to two points. This is impossible for the same reason as the previous case.

To summarize, we have found contradiction for all the cases,
and finished the proof of this lemma.
\end{proof}

Using Lemma \ref{lemma-7}, we can provides examples
of Katok metrics such that the estimates in Theorem \ref{main-thm} is sharp.

\begin{theorem}
Let $F$ be the Randers metrics on $S^n$ with $n>1$
satisfying (A1)-(A3) and one of (A4) and (A5),
then it has only finite orbits of prime closed geodesics.
Denote $H$ the subgroup of isometries preserving each
closed geodesic, $H_o$ its identity component, and $m=\dim H$. Then there exist exactly
$m$ geometrically distinct $\mathcal{B}_i$, such that the corresponding
${B}_i$'s are totally geodesic with non-trivial
$H_o$-actions.
\end{theorem}

The proof is a case-by-case discussion.
For each case,
it is not hard to calculate $G=I_o(M,F)$, $H_o$ and all the
orbits of prime closed geodesics.

For example, when $n_0>2$
and all $\gamma_i$'s are irrational numbers,
\begin{eqnarray*}
G&=&SO(n_0)\times U(n_1)\times\cdots\times U(n_k),\mbox{ and}\\
H&=& U(n_1)\times\cdots\times U(n_k),
\end{eqnarray*}
so we have $\dim H=k$.

When $1\leq i\leq k$,
$${B}_i=\{x=
(\mathbf{x}_0,\mathbf{z}_1,\ldots,\mathbf{z}_k)\in M\mbox{ with }\mathbf{x}_0=0\mbox{ and }\mathbf{z}_j=0\mbox{ when }j\neq i\}$$
is a homogeneous Randers sphere with exactly two orbits of
prime closed geodesics. It is isometrically imbedded in $(M,F)$
as a totally geodesic submanifold, because it is the fixed point set of the subgroup of $G$ with the $U(n_i)$-factor removed.
They provide all the different totally geodesic ${B}_i$'s with nontrivial $H_o$-actions.

There exists one more totally geodesic ${B}_{k+1}$
with a trivial $H_o$-action, i.e.
$${B}_{k+1}=\{x=
(\mathbf{x}_0,\mathbf{z}_1,\ldots,\mathbf{z}_k)\in M\mbox{ with }\mathbf{z}_1=\cdots=\mathbf{z}_k=0\}.$$
It is a standard unit sphere with only one orbit of closed geodesics.

By Lemma \ref{lemma-7}, no other closed geodesics can be found.

Summarizing all these observations, we see that this Randers sphere $(M,F)$ satisfies all the requirements in Theorem \ref{main-thm},
and the estimate in Theorem \ref{main-thm} for the number
of totally geodesic ${B}_i$'s is sharp.

The discussion for other cases is similar, so we skip
the details.

\end{document}